\documentclass[11pt]{article}
\usepackage{amssymb}
\usepackage{bbm}
\usepackage{amsfonts}
\usepackage{amsmath,latexsym,epsf,mathabx,tikz-cd,physics}
\usepackage{changepage}
\usepackage[all]{xy}
\usepackage{hyperref}
\usepackage{mathrsfs}
\usepackage{indentfirst}
\usepackage{dcolumn}
\usepackage{amsthm}
\numberwithin{equation}{section}
\usepackage{color}
\newdir{ >}{{}*!/-8pt/@{>}}
\begin{document}


\newtheorem{theorem}{\bf Theorem}[section]
\newtheorem{lemma}[theorem]{\bf Lemma}
\newtheorem{corollary}[theorem]{\bf Corollary}
\newtheorem{proposition}[theorem]{\bf Proposition}
\newtheorem{claim}[theorem]{\bf Claim}
\newtheorem{question}[theorem]{\bf Question}
\newtheorem{definition}[theorem]{\bf Definition}
\newtheorem{remark}[theorem]{\bf Remark}
\newtheorem{example}[theorem]{\bf Example}
\newtheorem{notation}[theorem]{\bf Notation}
\newtheorem{assertion}[theorem]{\bf Assertion}
\newtheorem{condition}[theorem]{\bf Condition}
\newcommand{\bt}{\begin{theorem}}
\newcommand{\et}{\end{theorem}}
\newcommand{\bl}{\begin{lemma}}
\newcommand{\el}{\end{lemma}}
\newcommand{\bd}{\begin{definition}}
\newcommand{\ed}{\end{definition}}
\newcommand{\bc}{\begin{corollary}}
\newcommand{\ec}{\end{corollary}}
\newcommand{\bp}{\begin{proof}}
\newcommand{\ep}{\end{proof}}
\newcommand{\bx}{\begin{example}}
\newcommand{\ex}{\end{example}}
\newcommand{\br}{\begin{remark}}
\newcommand{\er}{\end{remark}}
\newcommand{\be}{\begin{equation}}
\newcommand{\ee}{\end{equation}}
\newcommand{\ba}{\begin{align}}
\newcommand{\ea}{\end{align}}
\newcommand{\bn}{\begin{enumerate}}
\newcommand{\en}{\end{enumerate}}
\newcommand{\bcs}{\begin{cases}}
\newcommand{\ecs}{\end{cases}}
\def\jze{{ \begin{pmatrix} 0 & 0 \\ 1 & 0 \end{pmatrix}}\setcounter{equation}{0}}
\def\hjz#1#2{{ \begin{pmatrix} {#1} & {#2} \end{pmatrix}}\setcounter{equation}{0}}
\def\ljz#1#2{{  \begin{pmatrix} {#1} \\ {#2} \end{pmatrix}}\setcounter{equation}{0}}
\def\jz#1#2#3#4{{  \begin{pmatrix} {#1} & {#2} \\ {#3} & {#4} \end{pmatrix}}\setcounter{equation}{0}}

\title{\bf  Sincere silting modules and vanishing conditions\thanks{Supported by the National Natural Science Foundation of China (Grant No. 11771212) and a project funded by the Priority Academic Program Development of Jiangsu Higher
 Education Institutions}}

\footnotetext{E-mail:~liujifen24@163.com,~weijiaqun@njnu.edu.cn}
\smallskip
\author{\small Jifen Liu, Jiaqun Wei\\
\small Institute of Mathematics, School of Mathematics Sciences\\
\small Nanjing Normal University, Nanjing \rm210023 China}
\date{}
\maketitle
\baselineskip 15pt
%
%
%
\vskip 10pt%
\noindent {\bf Abstract}: Let $R$ be a perfect ring and $T$ be an $R$-module. We study characterizations of sincere modules, sincere silting modules and tilting modules in terms of various vanishing conditions. It is proved that $T$ is sincere silting if and only if $T$ is presilting satisfing the vanishing condition $\mathrm{KerExt}^{0\le i\le 1}_R(T,-)=0$, and that $T$ is tilting if and only if $\mathrm{Ker}\mathrm{Ext}^{0\leqslant i\leqslant 1}_{R}(T,-)=0$ and $\mathrm{Gen}T\subseteq \mathrm{Ker}\mathrm{Ext}^{1\leqslant i\leqslant 2}_{R}(T,-)$.
 As an application, we prove that a sincere silting $R$-module $T$ of finite projective dimension is tilting if and only if $\mathrm{Ext}^{i}_{R}(T,T^{(J)})=0$ for all sets $J$ and all integer $i\ge 1$. This not only extends a main result of Zhang \cite{Z} from finitely generated modules over Artin algebras to infinitely generated modules over more general rings, but also gives it a different proof without using the functor $\tau$ and Auslander-Reiten formula.





%
%
\vskip 30pt
\section{Introduction}
Tilting theory is important in the representation theory of Artin algebras. The study of classical tilting modules was initiated by Brenner and Butler \cite{BB} and
continued by Happel and Ringel \cite{HR}. Since then, the defining conditions for a classical tilting module have been relaxed to tilting modules of finite projective dimension \cite{M}, and further have  been relaxed to general rings and  infinitely generated modules by many authors such as Colby and Fuller   \cite{CF}, Colpi and Trlifaj \cite{CT}, Angeleri-H{\"u}gel and Coelho \cite{AC}, Bazzoni \cite{B0}, etc. In 2014, Adachi, Iyama and Reiten \cite{AIR} introduced $\tau $-tilting modules and
support $\tau $-tilting modules, which are new generalizations of classical tilting modules. Later, Angeleri-H{\"u}gel, Marks and Vit\'{o}ria \cite{AMV} introduced silting modules
over a ring as a general version of support $\tau$-tilting modules.

In this paper, we aim to characterize sincere silting modules and tilting modules (of projective dimension 1) in terms of various vanishing conditions. Vanishing conditions are useful in the study of tilting theory. For instance, a classical result says that, over a ring $R$, an $R$-module $T$ is tilting if and only if $T$ is pretilting satisfying a vanishing condition $\mathrm{KerExt}_R^{0\le i\le 1}(T,-)=0$, see Section 2 for details. We consider the relations between sincere modules and vanishing conditions and use them to obtain new characterizations of sincere silting modules and tilting modules. In particular, we prove that an $R$-module $T$ is sincere silting if and only if $T$ is presilting satisfying the vanishing condition $\mathrm{KerExt}_R^{0\le i\le 1}(T,-)=0$, and that an $R$-module $T$ is tilting if and only if $\mathrm{Ker}\mathrm{Ext}^{0\leqslant i\leqslant 1}_{R}(T,-)=0$ and $\mathrm{Gen}T\subseteq \mathrm{Ker}\mathrm{Ext}^{1\leqslant i\leqslant 2}_{R}(T,-)$.

Recently, Zhang \cite{Z} proved that, over an Artin algebra, every  self-orthogonal $\tau$-tilting module of finite projective dimension is tilting. As an application of our results, we show that every self-orthogonal sincere silting $R$-module $T$ of finite projective dimension is tilting. Since sincere silting modules are the general version of $\tau$-tilting modules, our result covers Zhang's result. Moreover, our method to prove the result is different from Zhang's and hence gives his result a new proof, without using the functor $\tau$ and Auslander-Reiten formula.

This paper is organized as follows: In Section $2$, we recall some definitions and preliminary results. In Section $3$, we give
characterizations of sincere modules, sincere silting modules and tilting modules in terms of vanishing conditions. We apply our result to obtain a generalization of Zhang's result in Section $4$.

\vskip 10pt

Throughout the paper, $R$ will always be a (two-side) perfect ring and modules will always be (not necessarily finitely generated) left $R$-modules. The category of all $R$-modules is denoted by $\mathrm{Mod}R$, and the subcategory of projective (resp., injective) modules is denoted by $\mathrm{Proj}R$ (resp., $\mathrm{Inj}R$).

Let $T$ be an $R$-module. The notation $\mathrm{Pd}T$ denotes the projective dimension of $T$, and $\mathrm{Add}T$ denotes the additive closure of $T$ consisting of all modules isomorphic to a direct summand of some direct sums of copies of $T$. We denote by $\mathrm{Gen}T$ the subcategory of $T$-generated modules (i.e., all epimorphic images of modules in $\mathrm{Add}T$), and by $\mathrm{Pres}T$ the subcategory of $T$-presented modules (i.e., all modules $M$ such that there is an exact sequence $T_1\to T_0\to M\to 0$ for some $T_1,T_0\in\mathrm{Add}T$).

For a subset $I\subseteq \bf{Z}$ and a class $\mathcal{C}\subseteq \mathrm{Mod}R$, we denote

    \centerline{$\mathcal{C}^{\bot_I} := \mathrm{Ker}\mathrm{Ext}^{I}_{R}(\mathcal{C},-) = \{M\ | \ \mathrm{Ext}^{i}_{R}(C,M)=0$ for all $i\in I$ and all $C\in \mathcal{C}\}$,}

    \noindent where $\mathrm{Ext}_R^0 := \mathrm{Hom}_R$. For instance,

       \centerline{${\mathcal{C}^{\bot_0}}: = \{M\ |\ \mathrm{Hom}_R(C,M)=0$ for all $C\in\mathcal{C}\}$,}


            \centerline{$\mathcal{C}^{\bot_{0\le i\le 1}}: = \{M\ |\ \mathrm{Ext}_R^i(C,M)=0$ for all $C\in\mathcal{C}$ and $i=0,1\}$.}

            \noindent The notation ${^{\bot_I}\mathcal{C}}$ is defined similarly.

We depict monomorphisms by $\rightarrowtail$ and epimorphisms by $\twoheadrightarrow $ in diagrams. For two morphisms $f: X\rightarrow Y$ and $g: Y\rightarrow Z$, the composition of $f$ and $g$ is denoted by $gf: X\rightarrow Z$.

\section{Preliminary}
In this section, we recall some basic definitions and results, which will be used later. One can refer to \cite{AF} for unexplained definitions and results.

\begin{definition} {\rm \cite{CT}\label{T1}}
	Let $R$ be a ring and $T$ an $R$-module. $T$ is said to be tilting if $\mathrm{Gen}T=T^{\bot_{1}}$, or equivalently, if $T$ satisfies the following three conditions:

	$(T1)$ $T$ has projective dimension at most one.

	$(T2)$ $\mathrm{Ext}^{1}_{R}(T,T^{(J)})=0$ for any set $J$.

	$(T3)$ There is an exact sequence $0\rightarrow R\rightarrow T_{0}\rightarrow T_{1}\rightarrow 0$ with $T_{0},T_{1}\in \mathrm{Add}T$.

By {\rm \cite[Definition ${^{\ast}3}$]{B0}}, the condition $(T3)$ can be replaced by the following vanishing condition

$(T3){^{\prime}}$ $T^{\bot_{0\le i\le 1}}=0$.
\end{definition}

Following \cite{AMV}, we say that $T$ is {\it partial tilting} if it is a direct summand of a tilting module. Moreover, we say that $T$ is {\it pretilting} if it satisfies conditions $(T1)$ and $(T2)$.

Let $R$ be a ring. For a morphism $\sigma$ in Mod$R$, following \cite{AMV}. we denote

\centerline{$D_{\sigma}:=\{ X\in \mathrm{Mod}R\ |\ \mathrm{Hom}_{A}(\sigma,X)\  \mathrm{is\  surjective}\}.$}

The following lemma lists some properties of the class $D_{\sigma}$.

\begin{lemma} {\rm\cite{AMV}}\label{D}
	Let $\sigma$ be a morphism in $\mathrm{Proj}R$ with the cokernel $T$. Then

	$(1)$ $D_{\sigma}$ is closed under epimorphic images, extensions and direct products.

	$(2)$ $D_{\sigma}\subseteq T^{\bot_{1}}$.
\end{lemma}

We recall the definition of silting modules as follows.

\begin{definition} {\rm{\cite{AMV}}}
	An $R$-module $T$ is partial silting if there is a projective presentation $\sigma$ of $T$ such that
	$D_{\sigma}$ is a torsion class and $T\in D_{\sigma}$.

	An $R$-module $T$ is silting if there is a projective presentation $\sigma$ of $T$ such that $\mathrm{Gen}T=D_{\sigma}$.

    In the case, one says that $T$ is a (partial) silting module with respect to $\sigma$.
\end{definition}

For convenience, we say that an $R$-module $T$ is {\it presilting} if there is a projective presentation $\sigma$ of $T$ such that $\mathrm{Gen}T\subseteq D_{\sigma}$.
    In the case, we also say that $T$ is presilting with respect to $\sigma$. It is easy to see that pretilting modules are just presilting modules of projective dimension not more than 1.

A partial silting module is actually a direct summand of some silting module \cite{AMV}. Thus, partial tilting modules are partial silting.
    It is easy to see that partial silting modules are always presilting. But the converse is not true in general, indeed, \cite[Example 3.9]{AMV} provides a infinite dimensional pretilting module over a finite dimensional algebra which is not partial silting. However, in case that the morphism $\sigma$ is a map in $\mathrm{proj}R$ (the subcategory of all finitely generated projective modules), these two are the same \cite{AMV}.

We remark that presilting modules are also called {\em (large) $\tau $-rigid} modules in \cite{BHPST}. Moreover, there is a characterization of presilting modules as follows.

\begin{lemma}{\rm\cite[Proposition 5.6]{BHPST}}\label{A}
    Let $R$ be a ring and $P_1\stackrel{\sigma}{\to} P_0\to T\to 0$ be exact such that $\mathrm{Im}\sigma$ has a projective cover. Then $T$ is presilting  if and only if $\mathrm{Gen}T\subseteq T^{\bot_1}$.
\end{lemma}

In general, not every module has a projective cover.
A ring $R$ is said to be {\em left} (resp., {\it right}) {\it perfect} if every left (resp., right) $R$-module has a projective cover. A ring is {\it (two-side) perfect} if it is both left perfect and right perfect. The following result shows that perfect rings have nice properties.

%
%
%
%
%
%
%
%

\begin{theorem} {\rm (\cite[Theorem P]{B})\label{AA}}
	Given a ring $R$ with the (Jacobson) radical $J=J(R)$, the following statements are equivalent:

	$(1)$ $R$ is left perfect;

	$(2)$ $R/J$ is semisimple and every non-zero left $R$-module contains a maximal submodule;

	$(3)$ $R$ has no infinite set of orthogonal idempotents and every non-zero right $R$-module has a nonzero socle.

\end{theorem}

Finally, we need the following results on minimal epimorphism and essential monomorphism.

\begin{lemma} {\rm (\cite[Corollary 5.13 and Corollary 5.17]{AF})\label{LC}}
	$(1)$ A monomorphism $f:L\to M$ is essential if and only if, for all homomorphisms (equivalently, epimorphisms) $h$, if $hf$ is monic, then $h$ is monic.

	$(2)$ An epimorphism $g:M\to N$ is minimal if and only if, for all homomorphisms (equivalently, monomorphisms) $h$ , if $gh$ is epimorphic, then $h$ is epimorphic.
\end{lemma}

\section{Sincere silting modules}

Let $R$ be a (two-side) perfect ring and $T$ be an $R$-module.
    In this section, we will study the characterizations of sincere modules, sincere silting modules and tilting modules in terms of various vanishing conditions.

We begin with considering the following four conditions:

\begin{verse}
{\bf (PT)} $T$ is sincere, i.e., $\mathrm{Hom}_{R}(P,T)\neq 0$ for all non-zero projective modules $P$.

{\bf (TI)} $T$ is cosincere, i.e., $\mathrm{Hom}_{R}(T,I)\neq 0$ for all non-zero injective modules $I$.

{\bf (TS)} For any non-zero simple $R$-module $S$, $S\in$ Subfac$(T)$, i.e., there is a diagram
$T\twoheadrightarrow Y\leftarrowtail S$ for some $Y$.

{\bf (ST)} For any non-zero simple $R$-module $S$, $S\in$ Facsub$(T)$, i.e., there is a diagram
$S\twoheadleftarrow M\rightarrowtail T$ for some $M$.
\end{verse}

\begin{lemma}\label{KT}
There are the following relations between the above four conditions.
	$$\xymatrix@C=1.5cm@R=1.5cm{
		(PT)\ar@2{<->}[r]&(TS)\ar@{=>}[d]\\
		(ST)\ar@2{<->}[r]\ar@{=>}[u]&(TI)
		}$$
\end{lemma}

\begin{proof}
	(PT) $\Rightarrow$ (TS). Let $S$ be simple. Since $R$ is left perfect, there is a projective cover $\pi :P\twoheadrightarrow S$. Since $T$ satisfies the condition (PT), there is $0\neq f\in \mathrm{Hom}_{R}(P,T)$. Now consider the following commutative diagram, where the below square is a pushout:
	$$\xymatrix{
		0\ar[d]&0\ar[d]\\
		\mathrm{Ker}f\ar[d]_{b}\ar[r]^{c}&\mathrm{Ker}h\ar[d]^{a}\\
		P\ar[d]_{f}\ar@{->>}[r]^{\pi}&S\ar[d]^{h}\\
		T\ar[r]_{g}&Y
	}$$
It is easy to see that $g$ and $c$ are epimorphisms by the properties of the pushout. Thus, $Y\in \mathrm{Gen}T$. Suppose that $h=0$. Then it is clear that
$a:\mathrm{Ker}h\to S$ is an isomorphism. Hence, $ac=\pi b$ is epimorphic. Since $\pi$ is a projective cover, $\pi$ is a minimal epimorphism. Then
one gets that $b$ is epimorphic by Lemma \ref{LC}(2). Since $fb=0$, we have $f=0$, this is a contradiction. So $h\neq 0$.
It is obvious that $h$ is a monomorphism since $S$ is simple. Consequently, $S\in$ Subfac$(T)$.

	(TS) $\Rightarrow$ (PT). Let $P$ be arbitrary projective $R$-module. Noting that $R$ is perfect,
	one sees that there exists a simple module $S$ such that $0\neq \pi^{\prime}:P\twoheadrightarrow S$ is epimorphic by Theorem \ref{AA}(2). Therefore, by (TS) and the projectivity of $P$, there exists a homomorphism $\alpha:P\to T$ such that the following diagram commutates.
	$$\xymatrix{
		P\ar@{-->}[r]^{\alpha}\ar@{->>}[d]_{\pi^{\prime} }&T\ar@{->>}[d]^{\pi}\\
		S\ar@{ >->}[r]_{i}&Y
	}$$
	Because $i$ is a monomorphism, $\pi^{\prime}$ is an epimorphism and $i\pi^{\prime}=\pi \alpha $, we obtain $\alpha \neq 0$. Consequently,
	$\mathrm{Hom}_{R}(P,T)\neq 0$.

	(ST) $\Rightarrow$ (PT). Let $P$ be arbitrary projective $R$-module. As above, there exists a simple module $S$ such that $0\neq \pi^{\prime}:P\twoheadrightarrow S$ is epimorphic.
	The condition (ST) implies that there is a diagram $S\twoheadleftarrow M\rightarrowtail T$ for some $M$. It follows from the projectivity of $P$ that there exists a morphism $h\neq 0$ such that $\pi^{\prime} =gh$.
	$$\xymatrix{
		P\ar@{->>}[d]_{\pi^{\prime} }\ar@{-->}[rd]^{h}& & \\
		S&M\ar@{->>}[l]^{g}\ar@{ >->}[r]_{f}&T.
	}$$
	 It is easy to see that $fh\neq 0$ since $f$ is monic. Therefore, $\mathrm{Hom}_{R}(P,T)\neq 0$.

	(TI) $\Rightarrow$ (ST). For any simple module $S$, there is $i:S\rightarrowtail E(S)$, where $E(S)$ is the injective envelope of $S$. By (TI), there is $0\neq f\in \mathrm{Hom}_{R}(T,E(S))$.
	Now consider the following commutative diagram with top square being a pullback.

	$$\xymatrix{
		M\ar[r]^{g}\ar[d]_{h}&T\ar[d]^{f}\\
		S\ar@{ >->}[r]^{i}\ar[d]_{b}&E(S)\ar[d]^{a}\\
		\mathrm{Coker}h\ar[r]_{c}\ar[d]&\mathrm{Coker}f\ar[d]\\
		0&0
	}$$
	It is easy to check that $g$ and $c$ are monomorphisms by the properties of the pullback. Suppose that $h=0$. Then it is obvious that
	$b:S\to \mathrm{Coker}h$ is an isomorphism. Consequently, $cb=ai$ is monic. Because $i$ is an essential monomorphism,
	we obtain that $a$ is monic by Lemma \ref{LC}(1). Noting that $af=0$, we have $f=0$, a contradiction. So $h\neq 0$.
	It follows that $h$ is an epimorphism since $S$ is simple. Consequently,  $S\in$ Facsub$(T)$.

	(ST) $\Rightarrow$ (TI). Let $I$ be injective. By Theorem \ref{AA}(3), there exists a simple module $S$ such that $0\neq i^{\prime}:S\rightarrowtail I$ is monic.
    The condition (ST) implies that there is a diagram $S\twoheadleftarrow M\rightarrowtail T$ for some $M$.
	Then, by the injectivity of $I$, there exists $\beta $ such that the following diagram commutates.

	$$\xymatrix{
		M\ar@{ >->}[r]^{i}\ar@{->>}[d]_{p}&T\ar@{-->}[d]^{\beta }\\
		S\ar@{ >->}[r]^{i^{\prime}}&I
	}$$
    Thus $i^{\prime}p=\beta i$. Because $p$ is epimorphic and $i^{\prime}$ is monic, we obtain $\beta \neq 0$. Hence,
	$\mathrm{Hom}_{R}(T,I)\neq 0$.

	(TS) $\Rightarrow$ (TI). Let $I$ be arbitrary injective module. As above, there exists a simple module $S$ such that $0\neq i^{\prime}:S\rightarrowtail I$ is monic.
	By the condition (TS), there is a diagram $T\twoheadrightarrow Y\leftarrowtail S$ for some $Y$.
	Noting that $I\in \mathrm{Inj}R$ and $f:S\rightarrowtail Y$ is monic, there exists $h\neq 0$ such that the following diagram commutates.
	$$\xymatrix{
		T\ar@{->>}[r]^{g}&Y\ar@{-->}[rd]_{h}&S\ar[d]^{i^{\prime}}\ar@{ >->}[l]_{f}\\
		 & &I
	}$$
	It is obvious that $hg\neq 0$ since $g$ is epimorphic. Consequently, $\mathrm{Hom}_{R}(T,I)\neq 0$.

\end{proof}

\begin{remark} (1) It always holds that {\rm Subfac}T = {\rm Facsub}T. Thus the conditions (TS) and (ST) are always equivalent.

	(2) Let $\varLambda $ be an Artin algebra and $T$ a finitely generated $\varLambda $-module. Then $T$ is
	sincere is equivalent to the fact that
	every simple module appears as a composition factor in $T$. 
\end{remark}

    In the following, we consider the relations between sincere modules and modules satisfying the vanishing condition $(T3){^{\prime}}$.

\begin{proposition}\label{P}
	If $T$ satisfies the vanishing condition $(T3){^{\prime}}$, then $T$ is sincere.
\end{proposition}

\begin{proof}
	By Lemma \ref{KT}, it suffices to prove that $\mathrm{Hom}_{R}(T,I)\neq 0$ for every non-zero injective module $I$.
	Suppose that there exists $0\neq I\in \mathrm{Inj}R$ such that $\mathrm{Hom}_{R}(T,I)=0$, then
	$I\in T^{\bot_{0\le i\le 1}}=0$, a contradiction. Thus, for any $0\neq I\in \mathrm{Inj}R$, we have
	$\mathrm{Hom}_{R}(T,I)\neq 0$.
\end{proof}

We note that the converse of Proposition \ref{P} is not true in general, as the following example shows.

\begin{example}\rm
	Let $\varLambda $ be an algebra given by the quiver $Q:1\stackrel{\alpha }\longleftarrow 2$. Then $P_{1}=1$ and $P_{2}=\mqty{2\\1}$ are projective modules.
	The AR-quiver of mod$\varLambda $ is as follows:
	$$\xymatrix@C=0.8cm@R=0.2cm{
		 &\mqty{2\\1}\ar[rd]\\
		1\ar[ru]& &2
	}$$

	Let $T=P_{2}=\mqty{2\\1}$. Then it is easy to check that $T$ is sincere, pd$_{\varLambda }T=0\leqslant 1$ and $\mathrm{Ext}^{1}_{\mathcal{\varLambda }}(T,T)=0$.
	But $0\neq P_{1}\in T^{\bot_{0\le i\le 1}}$.

	\vskip 10pt

	Note that it also gives an example of sincere pretilting modules but not tilting modules.
\end{example}

However, the following proposition shows that the converse of Proposition \ref{P} holds under certain conditions.

\begin{proposition}\label{PP}
	Assume that the $R$-module $T$ satisfies $\mathrm{Gen}T=\mathrm{Pres}T$. If $T$ is sincere, then
	$T$ satisfies the vanishing condition $(T3){^{\prime}}$.
\end{proposition}

\begin{proof}

    Taken any $M\in T^{\bot_{0\le i\le 1}}$, we first claim that $M\in (\mathrm{Gen}T)^{\bot_{1}}$ if $T$ satisfies the condition
	$\mathrm{Gen}T=\mathrm{Pres}T$.

    Indeed, for any $X\in \mathrm{Gen}T$, there exists a short exact sequence
	$$0\rightarrow X^{\prime}\rightarrow T_{X}\rightarrow X\rightarrow 0$$
	with $T_{X}\in \mathrm{Add}T$ and  $X^{\prime}\in \mathrm{Gen}T$ since $\mathrm{Gen}T=\mathrm{Pres}T$. Note that $\mathrm{Hom}_{R}(T,M)=0$. This implies
	that $\mathrm{Hom}_{R}(\mathrm{Gen}T,M)=0$. Applying $\mathrm{Hom}_{R}(-,M)$ to the above
	short exact sequence, we obtain the following long exact sequence
	$$\cdots\rightarrow \mathrm{Hom}_{R}(X^{\prime},M)\rightarrow \mathrm{Ext}_{R}^{1}(X,M)\rightarrow
	\mathrm{Ext}_{R}^{1}(T_{X},M)\rightarrow \cdots.$$
	Because $\mathrm{Hom}_{R}(X^{\prime},M)=0$ and $\mathrm{Ext}_{R}^{1}(T_{X},M)=0$, it
	follows that $\mathrm{Ext}_{R}^{1}(X,M)=0$.

	Now suppose that $M\neq 0$. Then there exists a simple module $S$ with $g:S\rightarrowtail M$ by Theorem \ref{AA}(3).
	Noting that $T$ is sincere, we have the following diagram, where $f$ and $h$ are promised by the condition (ST).

	$$\xymatrix{
		0\ar[r]&G\ar@{ >->}[r]^{h}\ar@{->>}[d]_{f}&T\ar[r]^(0.3){\pi}&C:=\mathrm{Coker}h\ar[r]&0\\
		 &S\ar@{ >->}[d]_{g}& & & \\
		 &M& & &
	}$$
	Applying $\mathrm{Hom}_{R}(-,M)$ to the short exact sequence $0\rightarrow G\rightarrow T\rightarrow C\rightarrow 0$, we get the long exact sequence
	$$0\rightarrow \mathrm{Hom}_{R}(C,M)\rightarrow \mathrm{Hom}_{R}(T,M)\rightarrow \mathrm{Hom}_{R}(G,M)\rightarrow \mathrm{Ext}_{R}^{1}(C,M)\rightarrow \cdots.$$
	Note that $\mathrm{Hom}_{R}(T,M)=0$ by assumption and that $\mathrm{Ext}_{R}^{1}(C,M)=0$, since $C\in \mathrm{Gen}T$ and $M\in (\mathrm{Gen}T)^{\bot_{1}}$.
	Therefore $\mathrm{Hom}_{R}(G,M)=0$ and then $gf=0$. On the other hand, since $g$ is monic and $f$ is epimorphic, we also have $gf\neq 0$, a contradiction. Consequently, $M=0$.
\end{proof}

Since a silting module $T$ always satisfies the condition $\mathrm{Gen}T=\mathrm{Pres}T$, we have the following corollary.

\begin{corollary}\label{C}
    Let $T$ be a silting $R$-module. Then $T$ is sincere if and only if $T$ satisfies the vanishing condition $(T3){^{\prime}}$.
\end{corollary}

The following result provides a generalization of the characterization of tilting modules in terms of the vanishing condition $(T3)^{\prime}$ to sincere silting modules.

\begin{theorem}\label{THM}
	$T$ is a sincere silting $R$-module if and only if $T$ is a presilting module and $T^{\bot_{0\le i\le 1}}=0$.
\end{theorem}

\begin{proof}
	$\Rightarrow )$ The assertion follows from the definition of silting modules and Corollary \ref{C}.

	$\Leftarrow )$ The condition $T^{\bot_{0\le i\le 1}}=0$ implies that $T$ is sincere by Proposition \ref{P}. Because $T$ is presilting,
    there exists a projective presentation
	$$P_{1}\stackrel{\sigma }\rightarrow P_{0}\rightarrow T\rightarrow 0$$
	such that $\mathrm{Gen}T\subseteq D_{\sigma}$. It suffices to show that $D_{\sigma}\subseteq \mathrm{Gen}T$.

    Noting that $R$ is a perfect ring, we obtain that
    $\mathrm{Gen}T\subseteq T^{\bot_{1}}$ by
    Lemma \ref{A}. Then it follows from \cite[Lemma 2.3]{AMV} that ($\mathrm{Gen}T,T^{\bot_{0}}$)
	is a torsion pair. Thus, for any $K\in D_{\sigma}$, we get the following canonical exact sequence
	$$0\rightarrow G\rightarrow K\rightarrow X\rightarrow 0$$
    with $G\in \mathrm{Gen}T$, $X\in T^{\bot_{0}}$. Because $K\in D_{\sigma}$ and $D_{\sigma}$ is closed under
    epimorphic images, we also get that
	$X\in D_{\sigma}\subseteq T^{\bot_{1}}$. So $X\in T^{\bot_{0\le i\le 1}}=0$.
    Consequently,
	$K\simeq G\in \mathrm{Gen}T$ and the proof is completed.
\end{proof}

In the following, we study characterizations of tilting modules in terms of various vanishing conditions.
    We note first that there is the following characterization of sincere modules.

\begin{proposition}\label{L}
	An $R$-module $T$ is sincere if and only if $^{\bot_{0}}(\mathrm{Gen}T)=0$.
\end{proposition}

\begin{proof}
	$\Leftarrow)$ Assume that $^{\bot_{0}}(\mathrm{Gen}T)=0$. Let $P$ be arbitrary projective $R$-module such that $\mathrm{Hom}_{R}(P,T)=0$. We only need to
    prove that $P=0$, by Lemma \ref{KT}.
	Indeed, for any $X\in \mathrm{Gen}T$, there is an exact sequence $T^{(J)}\rightarrow X\rightarrow 0$. Using that
	$P\in \mathrm{Proj}(R)$, we obtain that $\mathrm{Hom}_{R}(P,T^{(J)})\rightarrow \mathrm{Hom}_{R}(P,X)\rightarrow 0$
	is exact. But $\mathrm{Hom}_{R}(P,T^{(J)})\subseteq \mathrm{Hom}_{R}(P,T^{J})\simeq\mathrm{Hom}_{R}(P,T)^{J}=0$, 
	it follows that $\mathrm{Hom}_{R}(P,X)=0$, i.e., $P\in {^{\bot_{0}}}(\mathrm{Gen}T)$. Hence $P=0$.

	$\Rightarrow)$ Because $T$ is sincere, for any simple module $S$, we have a diagram $T\twoheadrightarrow Y\leftarrowtail S$ for some $Y$ by Lemma \ref{KT}.
	Thus $Y\in \mathrm{Gen}T$. For any $X\in {^{\bot_{0}}}(\mathrm{Gen}T)$, we get an exact sequence
	$0\rightarrow \mathrm{Hom}_{R}(X,S)\rightarrow \mathrm{Hom}_{R}(X,Y)$. It is obvious that $\mathrm{Hom}_{R}(X,S)=0$
	since $\mathrm{Hom}_{R}(X,Y)=0$. This shows that the socle of $X$ is zero. Therefore,  $X=0$ by Theorem \ref{AA}.
\end{proof}

The above result helps us get the following characterization of pretilting modules in case they are sincere.

\begin{proposition}\label{M}
	Let $T$ be a sincere $R$-module. Then $T$ is pretilting
	if and only if $\mathrm{Gen}T\subseteq T^{\bot_{1\le i\le 2}}$.
\end{proposition}

\begin{proof}
	$\Rightarrow )$ Easily.

	$\Leftarrow )$ Obviously it suffices to prove that $\mathrm{Pd}T\leq 1$. Noting that $R$ is a perfect ring, the condition $\mathrm{Gen}T\subseteq T^{\bot_{1\le i\le 2}}$
	implies that $T$ is presilting by Lemma \ref{A}. Then there exists a projective presentation
	$P_{1}\stackrel{\sigma}\rightarrow P_{0}\rightarrow T\rightarrow 0$ such that $\mathrm{Gen}T\subseteq D_{\sigma}$.
	Put $K:=\mathrm{Ker}\sigma$. It is easy to see that there are two short exact sequences
	\begin{equation*}
		0\rightarrow \mathrm{Im}\sigma \rightarrow P_{0}\rightarrow  T\rightarrow 0, \tag{1}
	\end{equation*}
	\begin{equation*}
		0\rightarrow K\rightarrow P_{1}\rightarrow \mathrm{Im}\sigma \rightarrow 0. \tag{2}
	\end{equation*}
	For any $X\in \mathrm{Gen}T$, applying $\mathrm{Hom}_{R}(-,X)$ to $(1)$, we get a long exact sequence
	$$\cdots\rightarrow \mathrm{Ext}^{1}_{R}(P_{0},X)\rightarrow \mathrm{Ext}^{1}_{R}(\mathrm{Im}\sigma,X)\rightarrow \mathrm{Ext}^{2}_{R}(T,X)\rightarrow \cdots.$$
	Noting that $\mathrm{Ext}^{1}_{R}(P_{0},X)=0$ and $\mathrm{Ext}^{2}_{R}(T,X)=0$ by assumption, we have that $\mathrm{Ext}^{1}_{R}(\mathrm{Im}\sigma,X)=0$. On the other hand, applying $\mathrm{Hom}_{R}(-,X)$ to $(2)$, we get a short exact sequence
	$$0\rightarrow \mathrm{Hom}_{R}(\mathrm{Im}\sigma,X)\rightarrow \mathrm{Hom}_{R}(P_{1},X)\rightarrow \mathrm{Hom}_{R}(K,X)\rightarrow \mathrm{Ext}^{1}_{R}(\mathrm{Im}\sigma,X)=0.$$
	Now we consider the following diagram
	$$\xymatrix{
		0\ar[r]&K\ar[r]^{i}\ar[d]_{f}&P_{1}\ar[r]^{\sigma}\ar@{-->}[ld]_{g}&P_{0}\ar[r]\ar@{-->}[lld]^{h}&T\ar[r]&0\\
		 &X& & & &
	}$$
	For any $f\in \mathrm{Hom}_{R}(K,X)$, there exists $g\in \mathrm{Hom}_{R}(P_{1},X)$ such that $f=gi$. Because
	$X\in \mathrm{Gen}T\subseteq D_{\sigma}$, we obtain that $\mathrm{Hom}_{R}(\sigma,X)$ is surjective. Then there is a morphism $h\in \mathrm{Hom}_{R}(P_{0},X)$
	such that $g=h\sigma$. It follows that $f=gi=h\sigma i=0$. This shows that $K\in {^{\bot_{0}}(\mathrm{Gen}T)}$. Consequently, $K=0$ by Proposition \ref{L} and then
    $\mathrm{Pd}T\leq 1$.
\end{proof}

In particular, we have the following corollary.

\begin{corollary}
	Let $T$ be a sincere presilting $R$-module. Then $T$ is a pretilting module if and only if
	$\mathrm{Gen}T\subseteq T^{\bot_2}$.
\end{corollary}

\begin{proof}
	$\Rightarrow )$ It is obvious.

	$\Leftarrow )$ Since $T$ is a presilting module, one gets $\mathrm{Gen}T\subseteq D_{\sigma}\subseteq T^{\bot_1}$. 
	Thus the assertion follows from Proposition \ref{M}.
\end{proof}

Now we can characterize tilting modules as follows.

\begin{theorem}\label{TM}
	$T$ is a tilting $R$-module if and only if $T^{\bot_{0\le i\le 1}}=0$ and
	$\mathrm{Gen}T\subseteq T^{\bot_{1\le i\le2}}$.
\end{theorem}

\begin{proof}
	$\Rightarrow )$ It is clear.

	$\Leftarrow )$ The condition $T^{\bot_{0\le i\le 1}}=0$ implies that $T$ is sincere by Proposition \ref{P}.
	Then $T$ is pretilting by Proposition \ref{M}. It follows that $T$ is a tilting module by Definition \ref{T1}.
\end{proof}

The following is an easy corollary.

\begin{corollary}\label{T}
    An $R$-module $T$ is tilting if and only if $T$ is sincere silting satisfying the condition
	$\mathrm{Gen}T\subseteq T^{\bot_2}$.
\end{corollary}

\begin{proof}
	$\Rightarrow)$ It is obvious by the definition of tilting modules.

	$\Leftarrow)$ Since $T$ is silting, one gets that $\mathrm{Gen}T\subseteq T^{\bot_1}$ by  the definition and Lemma \ref{D}. 
    Then the assertion follows from Corollary \ref{C} and Theorem \ref{TM}.
\end{proof}

\section{An application}

Recently, Zhang \cite{Z} considered the relations between tilting modules and self-orthogonal $\tau$-tilting modules over Arin algebras. He proved that a (finitely generated) self-orthogonal $\tau$-tilting module of finite projective dimension is tilting. This indeed provides a partial answer to the old rank question for tilting modules, first asked by Rickard and Schofield in \cite{RS}. In this short section, we will apply our results to obtain a generalization of Zhang's result to infinitely generated case. Moreover, since Zhang's proof heavily depends on the functor $\tau$ and Auslander-Reiten formula, it couldn't be transferred to infinitely generated case. Thus, our result also provides a new proof of his result, avoid using the functor $\tau$ and Auslander-Reiten formula.

\begin{lemma}\label{LL}
	Let $T$ be a self-orthogonal $R$-module, i.e., $\mathrm{Ext}^{i}_{R}(T,T^{(J)})=0$ for any set $J$ and integer $i\geqslant 1$, of finite projective dimension.
    If $\mathrm{Gen}T = \mathrm{Pres}T$,	then $\mathrm{Gen}T\subseteq T^{\bot_{i\ge 1}}$.
\end{lemma}

\begin{proof}
    Taken any $X\in\mathrm{Gen}T$, we need only to show that $X\in T^{\bot_{i\ge 1}}$.
	
    Since $\mathrm{Gen}T = \mathrm{Pres}T$, one easily obtains that, for any $X\in \mathrm{Gen}T$, there is a long exact sequence
	$$\cdots\rightarrow T^{(J_{1})}\stackrel{f_{1}}\rightarrow T^{(J_{0})}\stackrel{f_{0}}\rightarrow X\rightarrow 0$$
	with $\mathrm{Ker}f_{j}\in \mathrm{Gen}T$ for each $j\ge 0$.

    Now assume that $\mathrm{Pd}T= n < \infty$. Since $T$ is self-orthogonal, by applying the functor $\mathrm{Hom}_R(T,-)$ to the above exact sequence, we obtain that
	$$\mathrm{Ext}^{i}_{R}(T,X)\cong \mathrm{Ext}^{n+i}_{R}(T,\mathrm{Ker}f_{n-1})=0$$
	for integer $i\geqslant 1$ by the dimension shifting. It follows that $X\in T^{\bot_{i\ge 1}}$.
\end{proof}

\begin{theorem}\label{2T}
	Let $T$ be a sincere silting $R$-module of finite projective dimension. Then $T$ is a tilting module if and only if
	$\mathrm{Ext}^{i}_{R}(T,T^{(J)})=0$ for any set $J$ and integer $i\geqslant 1$.
\end{theorem}

\begin{proof}
	$\Rightarrow)$ It is obvious.

	$\Leftarrow )$ By Lemma \ref{LL}, we have that $\mathrm{Gen}T\subseteq T^{\bot_{i\ge 1}}\subseteq T^{\bot_2}$. Then the assertion follows from Corollary \ref{T}.
\end{proof}

\end{document}